\documentclass[11pt,reqno,oneside,a4paper]{article}
\pdfoutput=1
\usepackage[a4paper,includeheadfoot,left=25mm,right=25mm,top=00mm,bottom=20mm,headheight=20mm]{geometry}
\usepackage{amssymb,amsmath,amsthm}
\usepackage{xcolor,graphicx}
\usepackage{verbatim}
\usepackage{mathtools}
\usepackage{hyperref}
\usepackage{titling}
\usepackage{fancyhdr}
\newcommand{\runningtitle}{Running Title}
\newtheorem{thm}{Theorem}
\newtheorem{lem}[thm]{Lemma}

\newtheorem{prop}[thm]{Proposition}
\newtheorem{cor}[thm]{Corollary}

\theoremstyle{definition}

\newtheorem{crit}[thm]{Criterion}

\newtheorem*{prob*}{Problem}
\theoremstyle{remark}
\newtheorem{rmk}[thm]{Remark}

\newcommand{\ZZ}{\mathbb Z}              % The set of integers
\newcommand{\RR}{\mathbb R}              % The set of reals
\newcommand{\CC}{\mathbb C}              % The set of complex numbers
\renewcommand{\Re}{\operatorname*{Re}} 
\newcommand{\D}{\ensuremath{\,\mathrm{d}}}
\newcommand{\ri}{\ensuremath{\mathrm{i}}}
\newcommand{\re}{\ensuremath{\mathrm{e}}}
\newcommand{\la}{\ensuremath{\lambda}}
\renewcommand{\epsilon}{\varepsilon}
\renewcommand{\geq}{\geqslant}
\renewcommand{\leq}{\leqslant}
\providecommand{\clos}{\operatorname{clos}}
\newcommand{\abs}[1]{\left\lvert#1\right\rvert}
\newcommand{\sgn}{\operatorname{sgn}}

\newcommand{\Mspacer}{\;} %Spacer for below Matrix display functions
\newcommand{\Msup}[4]{#1_{#2\Mspacer#3}^{#4}} %Print a symbol with two subscripts and a superscript eg a matrix entry
\newcommand{\Msups}[5]{#1_{#2\Mspacer#3}^{#4\Mspacer#5}} %Print a symbol with two subscripts and two superscripts eg a matrix entry
\usepackage{scalerel}
\usepackage{stackengine}
\stackMath
\newcommand\reallywidecheck[1]{%
\savestack{\tmpbox}{\stretchto{%
  \scaleto{%
    \scalerel*[\widthof{\ensuremath{#1}}]{\kern-.6pt\bigwedge\kern-.6pt}%
    {\rule[-\textheight/2]{1ex}{\textheight}}%WIDTH-LIMITED BIG WEDGE
  }{\textheight}%
}{0.5ex}}%
\stackon[1pt]{#1}{\scalebox{-1}{\tmpbox}}%
}

\newcommand\reallywidehat[1]{%
\savestack{\tmpbox}{\stretchto{%
  \scaleto{%
    \scalerel*[\widthof{\ensuremath{#1}}]{\kern-.6pt\bigwedge\kern-.6pt}%
    {\rule[-\textheight/2]{1ex}{\textheight}}%WIDTH-LIMITED BIG WEDGE
  }{\textheight}%
}{0.5ex}}%
\stackon[1pt]{#1}{\tmpbox}%
}
\newcommand{\AckYNCSRP}[1]{#1 gratefully acknowledges support from Yale-NUS College summer research programme.}

\newcommand{\AckYNCSeed}[1]{#1 gratefully acknowledges support from Yale-NUS College seed grant IG21-SG101.}
\newcommand{\bigoh}{\mathcal{O}}
\newcommand{\lindecayla}{\bigoh\left(\la^{-1}\right)}

\numberwithin{equation}{section}

\providecommand{\argdot}{{}\cdot{}}
\author{
    B. Normatov\textsuperscript{\textasteriskcentered} and D. A. Smith\textsuperscript{\textasteriskcentered\textdagger} \\
    \footnotesize\textsuperscript{\textasteriskcentered} Division of Science, Yale-NUS College, Singapore \\ \footnotesize\textsuperscript{\textdagger} Department of Mathematics, National University of Singapore Singapore, \\
    \footnotesize\href{mailto:dave.smith@yale-nus.edu.sg}{\texttt{dave.smith@yale-nus.edu.sg}}
}
\title{The Airy equation with nonlocal conditions}
\renewcommand{\runningtitle}{Airy nonlocal}
\date{\today}

\begin{document}
\maketitle
\thispagestyle{fancy}

\begin{abstract}
    We study a third order dispersive linear evolution equation on the finite interval subject to an initial condition and inhomogeneous boundary conditions but, in place of one of the three boundary conditions that would typically be imposed, we use a nonlocal condition, which specifies a weighted integral of the solution over the spatial interval.
    Via adaptations of the Fokas transform method (or unified transform method), we obtain a solution representation for this problem.
    We also study the time periodic analogue of this problem, and thereby obtain long time asymptotics for the original problem with time periodic boundary and nonlocal data.
\end{abstract}

\section{Introduction} \label{sec:Introduction}

We study the following third order initial nonlocal value problem (INVP).
\begin{prob*}[Finite interval problem for the Airy equation with one nonlocal condition]
    \begin{subequations} \label{eqn:INVPwithBC}
    \begin{align}
        \label{eqn:INVPwithBC.PDE} \tag{\theparentequation.PDE}
        [\partial_t+\partial_x^3] u(x,t) &= 0 & (x,t) &\in (0,1) \times (0,\infty), \\
        \label{eqn:INVPwithBC.IC} \tag{\theparentequation.IC}
        u(x,0) &= U(x) & x &\in[0,1], \\
        \label{eqn:INVPwithBC.BC0} \tag{\theparentequation.BC0}
        u(1,t) &= h_0(t) & t &\in[0,\infty), \\
        \label{eqn:INVPwithBC.BC1} \tag{\theparentequation.BC1}
        u_x(1,t) &= h_1(t) & t &\in[0,\infty), \\
        \label{eqn:INVPwithBC.NC} \tag{\theparentequation.NC}
        \int_0^1 K(y)u(y,t)\D y &= h_2(t) & t &\in[0,\infty),
    \end{align}
    in which $U$, $h_k$, and $K$ are sufficiently smooth and $U$ is compatible with $h_k$ and $K$ in the sense that the boundary and nonlocal conditions hold when $u(\argdot,t)$ is replaced with $U$ on the left and the right sides are evaluated at $t=0$.
    \end{subequations}
\end{prob*}

We consider this problem to be of primarily theoretical interest in understanding the applicability of the Fokas transform method~\cite{Fok2008a} to problems with nonlocal conditions.
Indeed, problem~\eqref{eqn:INVPwithBC} is a third order generalization of the INVP for the heat equation already studied by one of the authors~\cite{MS2018a}, but this problem is of higher spatial order and is dispersive instead of dissipative.

The Fokas transform method (unified transform method) was developed by Fokas and collaborators in the late 1990s and early 2000s, initially as an inverse scattering transform method for integrable nonlinear equations on domains with a spatial boundary.
However, it was soon understood that a version of the method for linear evolution equations yielded novel results, including the solution of the finite interval initial boundary value problem (IBVP) for the Airy equation~\eqref{eqn:INVPwithBC.PDE}~\cite{Pel2005a}.
The method works by constructing an integral transform pair tailored specifically to the particular IBVP of interest.
The derivation requires a complex contour deformation applied to the exponential Fourier inversion theorem, followed by the implementation in spectral space of a map from boundary data to unknown boundary values, a D to N map.
See~\cite{DTV2014a} for a good pedagogical introduction to the Fokas transform method, including its application to IBVP for equation~\eqref{eqn:INVPwithBC.PDE}.
The Fokas transform method owes its success to the diagonalization property of the transform pair.
The transforms diagonalize the IBVP's spatial differential operator in a weaker sense than does the Fourier sine transform diagonalize the half line Dirichlet heat operator, but in a sense that is suficient to enable the integral transform method to work.
See~\cite{Smi2023a} for a fuller explanation of the method's applicability from the point of view of classical spectral theory.

In the field of linear evolution equations, the domain of applicability of the Fokas transform method has been extended to problems of higher spatial order~\cite{FP2001a}, problems with general linear two point boundary conditions~\cite{Smi2012a,FS2016a,PS2016a,ABS2022a}, problems for systems of equations~\cite{FP2005b,DGSV2018a,CGM2021a}, and problems with mixed partial derivatives~\cite{DV2013a}.
In the past decade, there has been interest in problems with more complex kinds of boundary conditions, such as interface problems~\cite{APSS2015a,DPS2014a,DS2015a,She2017a,DSS2016a}, oblique Robin problems~\cite{Man2012a}, multipoint problems~\cite{PS2018a}, and general interface problems on networks~\cite{SS2015a,ABS2022a}.
Some of this work reproduced earlier results, often with alternative solution representations more amenable to numerical evaluation, but many of the results were novel, especially for equations of third or higher spatial order.
As multipoint conditions can represent a discretized weighted mean of $u$ over the spatial interval, it is natural to ask whether their continuous analogue, nonlocal conditions such as~\eqref{eqn:INVPwithBC.NC}, may also be treated via the Fokas transform method.
In~\cite{MS2018a}, it was shown that this can indeed be done for the heat equation, but the question remained open for equations of higher spatial order or dispersive equations.
This work addresses both of those cases.

Specifically for the dispersive third order INVP~\eqref{eqn:INVPwithBC}, we obtain three main results:
\begin{enumerate}
    \item[(i)]{in theorem~\ref{thm:FullSoln.OneNonlocal.Solution.D}, a solution representation;}
    \item[(ii)]{in prop~\ref{prop:LongTime.HomogeneousINVP.LinDec}, long time decay of the solution where all boundary and nonlocal conditions are homogeneous;}
    \item[(iii)]{and, in theorem~\ref{thm:LongTime.MainResult}, long time asymptotics for the solution in the case that all boundary and nonlocal conditions are periodic.}
\end{enumerate}
In service of the first of these aims,~\S\ref{sec:FullSoln} is dedicated to an implementation of the Fokas transform method, with adaptations appropriate to this problem.
In~\S\ref{sec:Periodic}, we study a problem related to problem~\eqref{eqn:INVPwithBC} in which the data are time periodic with common period, but the initial condition is replaced with a time periodicity condition on $u$.
We use the ``$Q$ equation'' method~\cite{FL2012b,FvdW2021a,FPS2022a} to obtain a solution representation and criteria on its validity.
The greater part of~\S\ref{sec:LongTime} is dedicated to an asymptotic analysis of the solution representation for $u$, whereby the long time decay result is derived.
We then use the principle of linear superposition to decompose $u$ into a part satisfying the periodic problem and another part satisfying the homogeneous problem, obtaining the final theorem.

\section{Solution of the Airy equation with nonlocal conditions} \label{sec:FullSoln}

In this section, we derive an explicit integral representation of the solution of problem~\eqref{eqn:INVPwithBC}.
We use the Fokas transform method, adapting the arguments presented in~\cite{MS2018a}.

\subsection{Global relations and Ehrenpreis form}

Suppose that $u$ is a sufficiently smooth function satisfying the partial differential equation~\eqref{eqn:INVPwithBC.PDE}.
For $j\in\{0,1,2\}$, $y\in[0,1]$, and $\la\in\CC$, define
\begin{equation} \label{eqn:defn.f}
    f_j(\la;y,t) = \int_0^t \re^{-\ri\la^3 s} \partial_x^j u(y,s) \D s.
\end{equation}
In the case that $y\in\{0,1\}$, $f_j(\la;y,t)$ represents a time transform of a boundary value, which we refer to as a \emph{spectral boundary value}.
We use a hat $\hat\argdot$ to signify the spatial exponential Fourier transform of function $\argdot$ so that, for example,
\[
    \hat{u}(\la;t) = \int_0^1 \re^{-\ri\la x} u(x,t) \D x.
\]
Note that, in this Fourier transform, the zero extension of $u$ is taken beyond its spatial domain of definition.
Fix $0\leq y<z\leq 1$ and denote the exponential Fourier transform of the $[y,z]$ restriction of a function by appending the arguments $y,z$.
For example,
\[
    \hat{u}(\la;t;y,z) = \int_y^z \re^{-\ri\la x} u(x,t) \D x.
\]

We apply the $[y,z]$ restricted spatial Fourier transform to equation~\eqref{eqn:INVPwithBC.PDE} to obtain
\[
    0 = \reallywidehat{\left[u_t+u_{xxx}\right]}(\la;t;y,z).
\]
By linearity of the Fourier transform, smoothness of $u$ in time, and spatial integration by parts, we obtain
\begin{multline} \label{eqn:PreGR}
    0
    = \left[ \partial_t - \ri\la^3 \right] \hat{u}(\la;t;y,z)
    + \re^{-\ri\la z} \left[ \partial_{xx} u(z,t) + \ri\la\partial_xu(z,t) - \la^2u(z,t) \right] \\
    - \re^{-\ri\la y} \left[ \partial_{xx} u(y,t) + \ri\la\partial_xu(y,t) - \la^2u(y,t) \right].
\end{multline}

Suppose further that $u$ satisfies the initial condition~\eqref{eqn:INVPwithBC.IC}.
Solving the temporal first order linear ordinary differential equation~\eqref{eqn:PreGR}, we derive
\begin{multline} \label{eqn:GR}
    \re^{-\ri\la^3t} \hat{u}(\la;t;y,z)
    = \widehat{U}(\la;y,z)
    + \re^{-\ri\la y} \left[ f_2(\la;y,t) + \ri\la f_1(\la;y,t) - \la^2 f_0(\la;y,t) \right] \\
    - \re^{-\ri\la z} \left[ f_2(\la;z,t) + \ri\la f_1(\la;z,t) - \la^2 f_0(\la;z,t) \right],
\end{multline}
an equation we refer to as the \emph{global relation on $[y,z]$}.

To the global relation on $[0,1]$ we apply the inverse Fourier transform to find that
\begin{multline} \label{eqn:PreEhrenpreisForm}
    2\pi u(x,t)
    = \int_{-\infty}^\infty \re^{\ri\la x + \ri\la^3 t} \widehat{U}(\la) \D\la
    + \int_{-\infty}^\infty \re^{\ri\la x + \ri\la^3 t} \left[ f_2(\la;0,t) + \ri\la f_1(\la;0,t) - \la^2 f_0(\la;0,t) \right] \D\la \\
    - \int_{-\infty}^\infty \re^{\ri\la(x-1) + \ri\la^3 t} \left[ f_2(\la;1,t) + \ri\la f_1(\la;1,t) - \la^2 f_0(\la;1,t) \right] \D\la.
\end{multline}
Note that the integrals in equation~\eqref{eqn:PreEhrenpreisForm} are properly interpreted as Cauchy principal values, as they represent inverse Fourier transforms, and we cannot expect that $U$ or any of the other functions may be extended continuously by zero to $\RR$.
This Cauchy principal value interpretation of integrals shall be tacitly maintained for all formulae derived from equation~\eqref{eqn:PreEhrenpreisForm}.

We define the regions, each comprising a union of one or two open sectors,
\[
    D^\pm = \left\{ \la \in\CC^\pm: \Re(-\ri\la^3)<0 \right\}, \qquad\qquad E^\pm = \left\{ \la \in\CC^\pm: \Re(-\ri\la^3)>0 \right\},
\]
and adopt the convention that their boundaries, unions of rays $\re^{\ri j \pi/3}[0,\infty)$ for integer $j$, are oriented so that the region lies to the left of the ray.
We define further, for $R>0$,
\begin{equation} \label{eqn:defnDpmEpm}
    D^\pm_R = \left\{ \la \in D^\pm: \abs{\la}>R \right\}, \qquad\qquad E^\pm_R = \left\{ \la \in E^\pm: \abs{\la}>R \right\},
\end{equation}
also with their boundaries oriented so that the regions lie to the left.
Integration by parts in the definition of $f_j$ establishes that
\[
    \re^{\ri\la^3 t} \left[ f_2(\la;0,t) + \ri\la f_1(\la;0,t) - \la^2 f_0(\la;0,t) \right] = \lindecayla,
\]
uniformly in $\arg(\la)$, as $\la\to\infty$ within $\clos(E^+)$.
Hence, by Jordan's lemma and Cauchy's theorem,
\[
    \int_{\partial E^+} \re^{\ri\la x + \ri\la^3 t} \left[ f_2(\la;0,t) + \ri\la f_1(\la;0,t) - \la^2 f_0(\la;0,t) \right] \D\la = 0,
\]
and the second integral of equation~\eqref{eqn:PreEhrenpreisForm} may have its contour deformed from the real line to $\partial D_R^+$, for any $R>0$.
Similarly, the third integral of equation~\eqref{eqn:PreEhrenpreisForm} may have its contour deformed from the real line to $\partial D_R^-$, but in the opposite direction.
Therefore,
\begin{multline} \label{eqn:EhrenpreisForm.t}
    2\pi u(x,t)
    = \int_{-\infty}^\infty \re^{\ri\la x + \ri\la^3 t} \widehat{U}(\la) \D\la
    + \int_{\partial D_R^+} \re^{\ri\la x + \ri\la^3 t} \left[ f_2(\la;0,t) + \ri\la f_1(\la;0,t) - \la^2 f_0(\la;0,t) \right] \D\la \\
    + \int_{\partial D_R^-} \re^{\ri\la(x-1) + \ri\la^3 t} \left[ f_2(\la;1,t) + \ri\la f_1(\la;1,t) - \la^2 f_0(\la;1,t) \right] \D\la.
\end{multline}
By a similar Jordan's lemma argument, for all $j\in\{0,1,2\}$ and all $t'>t$,
\[
    \int_{\partial D_R^+} \re^{\ri\la x + \ri\la^3 t} \la^{2-j} \Big[ f_j(\la;0,t') - f_j(\la;0,t) \Big] \D\la = 0,
\]
and similarly for the integral over $\partial D_R^-$.
Therefore, equation~\eqref{eqn:EhrenpreisForm.t} may alternatively be expressed with any $t' \geq t$ substituted for each $t$ appearing as an argument of an $f_j$, but keeping the $t$ in the exponential kernels unchanged:
\begin{multline} \label{eqn:EhrenpreisForm}
    2\pi u(x,t)
    = \int_{-\infty}^\infty \re^{\ri\la x + \ri\la^3 t} \widehat{U}(\la) \D\la \\
    + \int_{\partial D_R^+} \re^{\ri\la x + \ri\la^3 t} \left[ f_2(\la;0,t') + \ri\la f_1(\la;0,t') - \la^2 f_0(\la;0,t') \right] \D\la \\
    + \int_{\partial D_R^-} \re^{\ri\la(x-1) + \ri\la^3 t} \left[ f_2(\la;1,t') + \ri\la f_1(\la;1,t') - \la^2 f_0(\la;1,t') \right] \D\la,
\end{multline}
which is known as the \emph{Ehrenpreis form}.
This is particularly convenient for efficiency of numerical computation when studying problem~\eqref{eqn:INVPwithBC} on a finite time interval, using $t'$ the final time, but may be inappropriate for studying long time asymptotics of solutions because $h_k$ need not be absolutely integrable on $[0,\infty)$.

Thusfar, because we have not made use of any boundary or nonlocal conditions, the results obtained above are similar to those one requires in the study of IBVP for the Airy equation on the finite interval.
The only difference is that the global relation~\eqref{eqn:GR} is presented for general subintervals of the spatial interval $[0,1]$.
This slight generalization of the global relation will be crucial in the following arguments.

\subsection{D to N map}
Equation~\eqref{eqn:EhrenpreisForm} is not an effective solution representation, because we do not have expressions for any of the six spectral boundary values $f_j$.
Note that this issue appears even in the case of IBVP;
any wellposed two point IBVP for the Airy equation must specify exactly three of the six $f_j$ appearing in the Ehrenpreis form~\eqref{eqn:EhrenpreisForm} or, more generally, exactly three linear combinations of the six.
One must, whether studying an INVP or an IBVP construct a map from the data $h_k$ to the presently unknown spectral boundary values $f_j$.
The fact that problem~\eqref{eqn:INVPwithBC} features a nonlocal condition in place of a boundary condition adds complexity to this D to N map, but is not the reason that such a map is required.

We seek expressions for each of the six spectral boundary values
\begin{equation}
    f_0(\la;0,t'), \qquad f_1(\la;0,t'), \qquad f_2(\la;0,t'), \qquad f_0(\la;1,t'), \qquad f_1(\la;1,t'), \qquad f_2(\la;1,t')
\end{equation}
in problem~\eqref{eqn:INVPwithBC}.
We apply the time transform
\[
    \phi(s) \mapsto \int_0^{t'} \re^{-\ri\la^3 s} \phi(s) \D s
\]
that was used in equation~\eqref{eqn:defn.f} to the boundary conditions~\eqref{eqn:INVPwithBC.BC0}--\eqref{eqn:INVPwithBC.BC1} to obtain expressions for two of these:
\begin{equation}
    f_0(\la;1,t') = \int_0^{t'} \re^{-\ri\la^3 s} h_0(s) \D s =: \tilde{h}_0(\la;t'),
    \qquad
    f_1(\la;1,t') = \int_0^{t'} \re^{-\ri\la^3 s} h_1(s) \D s =: \tilde{h}_1(\la;t').
\end{equation}

Let $\alpha=\re^{\ri2\pi/3}$, a primitive cube root of unity.
Observe that each of the functions $f_j$ are symmetric under the transformations $\la\mapsto\alpha^\ell\la$ for $\ell\in\{0,1,2\}$.
Using these substitutions in the global relation on $[0,1]$~\eqref{eqn:GR}, we could obtain three equations relating the remaining four unknown spectral boundary values.
If there were a third boundary condition in problem~\eqref{eqn:INVPwithBC}, then we might attempt to solve the resulting linear system for the remaining spectral boundary values, but there is no such third boundary condition, so this approach must be modified for the INVP.
The nonlocal condition~\eqref{eqn:INVPwithBC.NC} must play a role in the construction of the D to N map; if it did not, then problem~\eqref{eqn:INVPwithBC} could be solved without the nonlocal condition, which is false because problem~\eqref{eqn:INVPwithBC.PDE}--\eqref{eqn:INVPwithBC.BC1} is an IBVP known to be underspecified hence illposed~\cite{Pel2005a}.
Guided by these cogitations, we adapt the nonlocal condition and global relation so that they feature some common terms before employing the aforementioned $\la\mapsto\alpha^\ell\la$ symmetries.

We apply the same time transform from equation~\eqref{eqn:defn.f} to~\eqref{eqn:INVPwithBC.NC}, yielding
\begin{equation}
    \int_0^1 K(y) f_0(\la;y,t') \D y = \int_0^{t'} \re^{-\ri\la^3s} h_2(s) \D s = : \tilde{h}_2(\la;t').
\end{equation}
Henceforth, for efficiency of presentation, we suppress the $t'$ dependence of $f_j$ and $\tilde{h}_j$.
Instead of using the global relation on $[0,1]$, we use the global relation on $[y,1]$~\eqref{eqn:GR} at time $t'$, multiply each term by $\re^{\ri\la y}K(y)$, and integrate over $y\in[0,1]$, obtaining
\begin{equation} \label{eqn:nonlocalGR}
    \re^{-\ri\la} \widehat{K}(-\la) f_2(\la;1) - \ri\la \int_0^1 K(y) f_1(\la;y) \D y - \int_0^1 K(y) f_2(\la;y) \D y
    = N_1(\la) - \re^{-\ri\la^3t'} \nu_1(\la),
\end{equation}
which we refer to as the \emph{nonlocal global relation} and in which
\begin{align*}
    N_1(\la)   &= \la^2 \re^{-\ri\la} \widehat{K}(-\la) \tilde{h}_0(\la) - \ri\la \re^{-\ri\la} \widehat{K}(-\la) \tilde{h}_1(\la) - \la^2 \tilde{h}_2(\la)
    + \int_0^1 K(y) \re^{\ri\la y} \widehat{U}(\la;y,1) \D y, \\
    \nu_1(\la) &= \int_0^1 K(y) \re^{\ri\la y} \hat{u}(\la;t';y,1) \D y
\end{align*}
also have their dependence on $t'$ suppressed.
Note that the terms in $N_1$ are all expressed explicitly using the data of the problem, while $\nu_1$ involves $\hat{u}$, which is not a datum of the problem.
We beg the reader to tolerate the slight notational inconvenience of carrying around these two terms instead of combining them because of the benefit in emphasizing the separation of data and nondata.
Using the maps $\la\mapsto\alpha^j\la$ for $j\in\{0,1,2\}$, we obtain the linear system
\begin{equation*}
    \begin{pmatrix}
        1 & 1 & \re^{-\ri\la}\widehat{K}(-\la) \\
        \alpha & 1 & \re^{-\ri\alpha\la}\widehat{K}(-\alpha\la) \\
        \alpha^2 & 1 & \re^{-\ri\alpha^2\la}\widehat{K}(-\alpha^2\la)
    \end{pmatrix}
    \begin{pmatrix}
        -\ri\la\int_0^1 K(y) f_1(\la;y) \D y \\
        -\int_0^1 K(y) f_2(\la;y) \D y \\
        f_2(\la;1)
    \end{pmatrix}
    =
    \begin{pmatrix} N_1(\la) \\ N_1(\alpha\la) \\ N_1(\alpha^2\la) \end{pmatrix}
    - \re^{-\ri\la^3t'} \begin{pmatrix} \nu_1(\la) \\ \nu_1(\alpha\la) \\ \nu_1(\alpha^2\la) \end{pmatrix}.
\end{equation*}

We solve this system to obtain an expression for $f_2(\la;1)$.
We could also determine expressions for the other two entries in the vector of unknowns, but that is unnecessary because they do not appear in the Ehrenpreis form~\eqref{eqn:EhrenpreisForm}.
Via Cramer's rule, we find
\[
    f_2(\la;1) = \frac{1}{\Delta(\la)}\sum_{j=0}^2 \alpha^j N_1(\alpha^j\la) - \frac{\re^{-\ri\la^3t'}}{\Delta(\la)}\sum_{j=0}^2 \alpha^j \nu_1(\alpha^j\la),
\]
in which the determinant of the system is, up to multiplication by $\ri\sqrt{3}$,
\[
    \Delta(\la) = \sum_{j=0}^2 \alpha^j \re^{-\ri\alpha^j\la} \widehat{K}(-\alpha^j\la).
\]
We now have expressions for three of the six spectral boundary values, albeit with one depending on $\nu_1$.
The linear combination of spectral boundary values that appears in the integral along $\partial D_R^-$ of the Ehrenpreis form~\eqref{eqn:EhrenpreisForm} is
\begin{multline} \label{eqn:RightSpectralBoundaryValues}
    f_2(\la;1) + \ri\la f_1(\la;1) - \la^2 f_0(\la;1)
    = \frac{1}{\Delta(\la)}\sum_{j=0}^2 \alpha^j N_1(\alpha^j\la) + \ri\la \tilde{h}_1(\la;t) - \la^2 \tilde{h}_0(\la;t) \\
    - \frac{\re^{-\ri\la^3t'}}{\Delta(\la)}\sum_{j=0}^2 \alpha^j \nu_1(\alpha^j\la).
\end{multline}

The global relation on $[0,1]$~\eqref{eqn:GR} is
\begin{equation} \label{eqn:LeftSpectralBoundaryValues}
    f_2(\la;0) + \ri\la f_1(\la;0) - \la^2 f_0(\la;0)
    =
    N_0(\la)
    - \re^{-\ri\la^3t'} \nu_0(\la),
\end{equation}
where
\begin{align*}
    N_0(\la) &= \re^{-\ri\la} \left[ \frac1{\Delta(\la)} \sum_{j=0}^2 \alpha^j N_1(\alpha^j\la) + \ri\la \tilde{h}_1(\la) - \la^2\tilde{h}_0(\la) \right] - \widehat{U}(\la), \\
    \nu_0(\la) &= \frac{\re^{-\ri\la}}{\Delta(\la)} \sum_{j=0}^2 \alpha^j \nu_1(\alpha^j\la) - \hat{u}(\la;t').
\end{align*}
As above, $N_0$ contains the data of the problem and $\nu_0$ contains nondata, and both have their dependence on $t'$ suppressed.
Equation~\eqref{eqn:LeftSpectralBoundaryValues} provides precisely the linear combination of spectral boundary values that appears in the integral along $\partial D_R^+$ of the Ehrenpreis form.

\subsection{Solution of the INVP}

By substituting formulae~\eqref{eqn:RightSpectralBoundaryValues} and~\eqref{eqn:LeftSpectralBoundaryValues} into the Ehrenpreis form~\eqref{eqn:EhrenpreisForm}, one obtains an expression for the solution $u(x,t)$,
but it depends on both the data $N_0,N_1$ and the nondata $\nu_0,\nu_1$.
We aim to show that the terms involving nondata contribute nothing to the solution.
The main tools are the following lemmata.

\begin{lem} \label{lem:FullSoln.OneNonlocal.ZerosLocus}
    Suppose that $K$ has bounded total variation and $K$ is continuous and nonzero at $0$.
    There exists a finite $R>0$ such that there are no zeros of $\Delta$ in $\clos(D_R^+ \cup D_R^-)$.
\end{lem}

\begin{lem} \label{lem:FullSoln.OneNonlocal.NondataDecays}
    Suppose that $K$ has bounded total variation and $K$ is continuous and nonzero at $0$.
    For all $t' \geq t>0$,
    \[
        \re^{-\ri\la^3(t'-t)} \nu_0(\la;t') = \bigoh(1),
    \]
    uniformly in $\arg(\la)$, as $\la\to\infty$ within $\clos(D^+_R)$.
    Similarly, for all $t' \geq t>0$,
    \[
        \re^{-\ri\la^3(t'-t)}\frac{1}{\Delta(\la)}\sum_{j=0}^2 \alpha^j \nu_1(\alpha^j\la)
        = \lindecayla,
    \]
    uniformly in $\arg(\la)$, as $\la\to\infty$ within $\clos(D^-_R)$.
\end{lem}

In the Ehrenpreis form~\eqref{eqn:EhrenpreisForm}, we make substitutions for the spectral boundary values using formulae~\eqref{eqn:RightSpectralBoundaryValues} and~\eqref{eqn:LeftSpectralBoundaryValues}, to obtain
\begin{multline} \label{eqn:FullSoln.OneNonlocal.SolutionWithNondata}
    2\pi u(x,t)
    = [\mbox{data}]
    - \int_{\partial D_R^+} \re^{\ri\la \frac x2 - \ri\la^3 (t'-t)} \nu_0(\la;t')\re^{\ri\la \frac x2} \D\la \\
    - \int_{\partial D_R^-} \re^{\ri\la(x-1) - \ri\la^3 (t'-t)} \frac{1}{\Delta(\la)}\sum_{j=0}^2 \alpha^j \nu_1(\alpha^j\la;t') \D\la.
\end{multline}
By lemmata~\ref{lem:FullSoln.OneNonlocal.ZerosLocus}~and~\ref{lem:FullSoln.OneNonlocal.NondataDecays}, Jordan's lemma, and Cauchy's theorem, the two displayed integrals on the right of equation~\eqref{eqn:FullSoln.OneNonlocal.SolutionWithNondata} both evaluate to zero.
Note that, to justify the application of Jordan's lemma in the first integral, we used
\[
    \re^{\ri\la \frac x2} = \bigoh\left( \exp\left[-x\abs\la\frac{\sqrt3}2\right] \right),
\]
uniformly in $\arg(\la)$, as $\la\to\infty$ within $\clos(D^+_R)$,
and the other $\re^{\ri\la \frac x2}$ factor is used as the kernel for Jordan's lemma.
This justifies the following theorem.

\begin{thm} \label{thm:FullSoln.OneNonlocal.Solution.D}
    Suppose that problem~\eqref{eqn:INVPwithBC} has solution $u(x,t)$, that $U$ and $h_k$ are piecewise continuous, that $K$ has bounded total variation, and that $K$ is continuous and nonzero at $0$.
    Then, for all $t'\geq t$,
    \begin{multline} \label{eqn:FullSoln.DtoN.OneNonlocal.Solution.D}
        2\pi u(x,t)
        = \int_{-\infty}^\infty \re^{\ri\la x + \ri\la^3 t} \widehat{U}(\la) \D\la
        + \int_{\partial D_R^+} \re^{\ri\la x + \ri\la^3 t} N_0(\la;t') \D\la \\
        + \int_{\partial D_R^-} \re^{\ri\la(x-1) + \ri\la^3 t} \left[ \frac{1}{\Delta(\la)}\sum_{j=0}^2 \alpha^j N_1(\alpha^j\la;t') + \ri\la \tilde{h}_1(\la;t') - \la^2 \tilde{h}_0(\la;t') \right] \D\la.
    \end{multline}
\end{thm}

By their definition~\eqref{eqn:defnDpmEpm}, $E_R^+\cup E_R^-$ and $D_R^+\cup D_R^-$ have the same boundaries, except on the circle $C(0,R)$, but the boundaries are oppositely oriented.
Note that two of the four semiinfinite components of $\partial D_R^-$ are oppositely oriented but coincident with two of the four semiinfinite components of $\partial E_R^+$, and these lie along the real line.
It follows immediately from the definition of $N_0$ that, for all $\la\in\RR$,
\[
    \widehat{U}(\la)
    = - N_0(\la;t') + \re^{-\ri\la} \left[ \frac{1}{\Delta(\la)}\sum_{j=0}^2 \alpha^j N_1(\alpha^j\la;t') + \ri\la \tilde{h}_1(\la;t') - \la^2 \tilde{h}_0(\la;t') \right].
\]
making this substitution in the first integral of equation~\eqref{eqn:FullSoln.DtoN.OneNonlocal.Solution.D} and perturbing the contour away from the real line around $C(0,R)$, we arrive at the following corollary.

\begin{cor} \label{cor:FullSoln.OneNonlocal.Solution.E}
    Under the criteria of theorem~\ref{thm:FullSoln.OneNonlocal.Solution.D}, for all $t'\geq t$,
    \begin{multline} \label{eqn:FullSoln.DtoN.OneNonlocal.Solution.E}
        2\pi u(x,t)
        = - \int_{\partial E_R^+} \re^{\ri\la x + \ri\la^3 t} N_0(\la;t') \D\la \\
        - \int_{\partial E_R^-} \re^{\ri\la(x-1) + \ri\la^3 t} \left[ \frac{1}{\Delta(\la)}\sum_{j=0}^2 \alpha^j N_1(\alpha^j\la;t') + \ri\la \tilde{h}_1(\la;t') - \la^2 \tilde{h}_0(\la;t') \right] \D\la.
    \end{multline}
\end{cor}

\begin{proof}[Proof of lemma~\ref{lem:FullSoln.OneNonlocal.ZerosLocus}]
    This proof follows the arguments of~\cite{Lan1931a}.
    For notational convenience, we define $\kappa(y)=K(1-y)$ so that $\hat\kappa(\la)=\re^{-\ri\la}\widehat{K}(-\la)$.
    
    It is immediate from the definition that $\Delta(\alpha^k\la)=\alpha^{-k}\Delta(\la)$, so the zeros of $\Delta$ are arranged symmetrically according to rotation by $2\pi/3$.
    Consider $\la$ such that $\ri\la=a+b\ri$ with $a,b\gg0$.
    In this region, we will show that the term
    \(
        \alpha \hat\kappa(\alpha\la)
    \)
    is nonzero and dominates the other two terms of $\Delta$, from which it follows that $\Delta$ has no zeros in this region.
    The argument is very similar in the $a,-b\gg0$ region, but with the
    \(
        \alpha^2 \hat\kappa(\alpha^2\la)
    \)
    dominant instead.
    Therefore, outside some disc $B(0,R)$, the zeros of $\Delta$ are confined to semistrips of some finite width $2w$ about the rays $-\ri\alpha^j\RR^+$, which proves the lemma.
    
    Suppose $\ri\la=a+b\ri$ with $a,b\gg0$ and insist $0<\delta\ll1$ so that $\kappa$ is continuous on at least $[1-\delta,1]$.
    Then
    \[
        \alpha \hat\kappa(\alpha\la) = \frac{\ri}{\la} \left( \kappa(1)\re^{-\ri\alpha\la} - \kappa(1-\delta)\re^{-\ri\alpha\la(1-\delta)} - \int_{1-\delta}^1 \re^{-\ri\alpha\la y} \D\kappa(y) \right) + \alpha\int_{0}^{1-\delta} \re^{-\ri\alpha\la y} \kappa(y)\D y.
    \]
    Hence
    \begin{multline*}
        \abs{\alpha\hat\kappa(\alpha \la)}
        \geq
        \frac{\re^{\frac{a+\sqrt3b}{2}}}{\sqrt{a^2+b^2}}\left[ \kappa(1)
        - \re^{-\delta\frac{a+\sqrt3b}{2}}\kappa(1-\delta) - \max_{1-\delta\leq y_1<y_2\leq1}\abs{\kappa(y_2)-\kappa(y_1)}  \right. \\
        \left. - \re^{-\delta\frac{a+\sqrt3b}{2}} \sqrt{a^2+b^2} \max_{y\in[0,1-\delta]}\kappa(y)
        \right].
    \end{multline*}
    Because $\kappa$ is of bounded variation, the second maximum exists.
    Because $\kappa$ is continuous at $1$, the first maximum approaches $0$ as $\delta\to0$.
    Hence, there exists $R,w>0$ such that, for all $a>R$ and $b>w$, there exists $\delta>0$ for which the $\kappa(1)$ term dominates the others:
    \[
        \abs{\alpha\hat\kappa(\alpha \la)}
        \geq
        \frac{\re^{\frac{a+\sqrt3b}{2}}}{\sqrt{a^2+b^2}}\left[ \frac{\kappa(1)}2\right],
    \]
    say.
    But
    \[
        \abs{\Delta(\la)} \geq \abs{\alpha\hat\kappa(\alpha \la)} - \abs{\hat\kappa(\la)} -  \abs{\alpha^2\hat\kappa(\alpha^2 \la)},
    \]
    and the latter two terms are bounded by
    \[
        \max_{y\in[0,1]}\abs{\kappa(y)} \qquad \mbox{and} \qquad \max_{y\in[0,1]}\abs{\kappa(y)}\max\left\{\re^{\frac{a-\sqrt3b}{2}},1\right\},
    \]
    respectively.
    Hence, possibly after further increasing one or both of $R$ and $w$, it must be that $\Delta(\la)\neq0$.
\end{proof}

\begin{proof}[Proof of lemma~\ref{lem:FullSoln.OneNonlocal.NondataDecays}]
    Note that, if $t'\geq t$ and $\la\in\clos(D^\pm)$, then $\Re[-\ri\la^3(t'-t)]\leq0$, so the exponential factors $\exp[-\ri\la^3(t'-t)]=\bigoh(1)$, uniformly in $\arg(\la)$ and may be discounted.
    
    Note that $D_R^+ \cup D_R^-$ has three connected components, of which one is $D_R^+$ and the other two comprise $D_R^-$.
    We label these components with the subscript $1, 2, 3$ counting anticlockwise from the positive real axis so that $D_{R \Mspacer 1}^+=D_R^+$.
    With $\kappa$ as defined in the proof of lemma~\ref{lem:FullSoln.OneNonlocal.ZerosLocus}, using the criteria on $K$ to integrate by parts,
    \begin{multline*}
        \Delta(\la) = \frac\ri\la\left[ \kappa(1)\re^{-\ri\la} - \kappa(1-\delta)\re^{-\ri\la(1-\delta)} - \int_{1-\delta}^1\re^{-\ri\la y}\D\kappa(y) \right]
        + \int_0^{1-\delta}\re^{-\ri\la y} \kappa(y) \D y \\
        + \sum_{j=1}^2 \alpha^j \int_0^1 \re^{-\ri\alpha^j\la y} \kappa(y) \D y,
    \end{multline*}
    where $0<\delta\ll1$ is such that $\kappa$ is continuous on at least $[1-\delta,1]$.
    Consider $\la\to\infty$ from within $\clos(D^+)$ in this expression.
    Both the integrals in the final sum are $\bigoh(1)$ and both the second term of the bracket and the integral from $0$ to $1-\delta$ are $\bigoh(\exp(-\ri\la[1-\delta]))$.
    The third term in the bracket is bounded by
    \[
        \abs{\re^{-\ri\la}} \max_{1-\delta\leq y_1 < y_2\leq1}\abs{\kappa(y_2)-\kappa(y_1)},
    \]
    and the maximum has limit zero as $\delta\to0$ because $\kappa$ is continuous on $[1-\delta,1]$.
    Therefore, by fixing $\delta>0$ small enough, we can ensure that this maximum is no greater than $\kappa(1)/2$.
    Then the leading order term has nonzero coefficient, and all other terms are relatively decaying.
    Using the rotational symmetry of $\Delta$, we can obtain similar estimates for the behaviour of $\Delta$ in other sectors.
    Indeed,
    \[
        \Delta(\la) =
        \begin{cases}
            \Theta\left(\re^{-\ri\la}/\la\right) & \mbox{as } \la\to\infty \mbox{ from within } \clos\left(D_R^+\right), \\
            \Theta\left(\re^{-\ri\alpha^2\la}/\la\right) & \mbox{as } \la\to\infty \mbox{ from within } \clos\left(D_{R\Mspacer 2}^-\right), \\
            \Theta\left(\re^{-\ri\alpha\la}/\la\right) & \mbox{as } \la\to\infty \mbox{ from within } \clos\left(D_{R\Mspacer 3}^-\right).
        \end{cases}
    \]
    
    The dominant term in
    \begin{equation} \label{eqn:FullSoln.OneNonlocal.NondataDecays.proof1}
        \sum_{j=0}^2 \alpha^j \nu_1(\alpha^j\la)
    \end{equation}
    is the $j=2$ term in $\clos\left(D_{R\Mspacer 2}^-\right)$ and the $j=1$ term in $\clos\left(D_{R\Mspacer 3}^-\right)$.
    However, the symmetry of expression~\eqref{eqn:FullSoln.OneNonlocal.NondataDecays.proof1} means that we may instead (and more notationally conveniently so) show decay of only $\nu_1(\la)/\Delta(\la)$ as $\la\to\infty$ from within $\clos\left(D_R^+\right)$ and thereby conclude the latter claim of the lemma.
    
    Select any $0<\delta\ll1$ for which $K$ is continuous on $[0,\delta]$.
    Integration by parts implies
    \begin{multline} \label{eqn:FullSoln.OneNonlocal.NondataDecays.proof1a}
        \nu_1(\la)
        =
        \frac{1}{\ri\la} \bigg[ K(\delta) \re^{\ri\la\delta} \hat{u}(\la;t';\delta,1) - K(0) \hat{u}(\la;t';0,1) - \int_0^\delta \re^{\ri\la y} \hat{u}(\la;t';y,1) \D K(y) \\
        + \int_0^\delta K(y)u(y;t') \D y \bigg] + \int_\delta^1 K(y) \re^{\ri\la y} \hat{u}(\la;t';y,1) \D y.
    \end{multline}
    Integrating by parts, for all $y\in[0,1)$,
    \begin{align}
        \notag
        \abs{\hat{u}(\la;t';y,1)}
        &= \abs{\frac\ri\la\left[ \re^{-\ri\la}u(1,t') - \re^{-\ri\la y}u(y,t') - \int_y^1 \re^{-\ri\la z} u_x(z,t')\D z \right]} \\
        \notag
        &\leq \frac1{\abs\la} \left[ \abs{u(1,t')}\abs{\re^{-\ri\la}} + \abs{u(y,t')}\abs{\re^{-\ri\la y}} + (1-y)\abs{\re^{-\ri\la}} \max_{z\in[y,1]}\abs{u_x(z,t')} \right] \\
        &\leq M\abs{\frac{\re^{-\ri\la}}\la},
        \label{eqn:FullSoln.OneNonlocal.NondataDecays.proof1b}
    \end{align}
    for some $M>0$, which can be chosen uniformly in $y$.
    This immediately implies that the first and second terms in the bracket on the right of equation~\eqref{eqn:FullSoln.OneNonlocal.NondataDecays.proof1a} are $\bigoh(\la^{-1}\re^{-\ri\la})$.
    But also
    \[
        \abs{\int_0^\delta \re^{\ri\la y} \hat{u}(\la;t';y,1) \D K(y)} \leq M\abs{\frac{\re^{-\ri\la}}\la} V_0^\delta(K),
    \]
    where $V_0^\delta(K)$ represents the total variation of $K$ over $[0,\delta]$, so the third term in the bracket is also $\bigoh(\la^{-1}\re^{-\ri\la})$.
    The fourth term in the bracket is independent of $\la$.\footnote{This term is known by equation~\eqref{eqn:INVPwithBC.NC} to be $h_2(t')$, but we need not use that fact here.}
    The final term on the right of equation~\eqref{eqn:FullSoln.OneNonlocal.NondataDecays.proof1a} is $\bigoh(\re^{-\ri\la(1-\delta)})$.
    Hence, overall, $\nu_1(\la)/\Delta(\la) = \bigoh(\la^{-1})$.
    
    It remains only to establish the first claim of the lemma, in which we study $\la\to\infty$ within $\clos\left(D_R^+\right)$.
    Note that
    \[
        \frac{\re^{-\ri\la}}{\Delta(\la)} \int_0^1 K(y) \re^{\ri\la y} \D y = 1 + \bigoh(\la^{-1}).
    \]
    Therefore,
    \begin{align}
        \notag
        \nu_0(\la;t')
        &= \frac{\re^{-\ri\la}}{\Delta(\la)} \left( \sum_{j=0}^2 \alpha^j \nu_1(\alpha^j\la) - \int_0^1 K(y) \re^{\ri\la y} \hat{u}(\la;t') \D y \right) + \lindecayla \\
        \notag
        &= \frac{\re^{-\ri\la}}{\Delta(\la)} \left( \sum_{j=1}^2 \alpha^j \int_0^1 K(y) \re^{\ri\alpha^j\la y} \hat{u}(\alpha^j\la;t';y,1) \D y \right. \\ &\hspace{13em} \left. \vphantom{\sum_{j=1}^2} - \int_0^1 K(y) \re^{\ri\la y} \hat{u}(\la;t';0,y) \D y \right) + \lindecayla.
        \label{eqn:FullSoln.OneNonlocal.NondataDecays.proof2}
    \end{align}
    Using similar calculations to those justifying inequalities~\eqref{eqn:FullSoln.OneNonlocal.NondataDecays.proof1b}, we find that
    \begin{align*}
        \re^{\ri\alpha^j\la y} \hat{u}(\alpha^j\la;t';y,1) &= \bigoh(\la^{-1}), \\
        \re^{\ri\la y} \hat{u}(\alpha^j\la;t';0,y) &= \bigoh(\la^{-1}),
    \end{align*}
    both uniformly in $y$.
    Hence both integrals in equation~\eqref{eqn:FullSoln.OneNonlocal.NondataDecays.proof2} are $\bigoh(\la^{-1})$, and $\nu_0=\bigoh(1)$.
\end{proof}

\begin{rmk}
    If we assume that $K$ is not just of bounded variation but continuously differentiable, then we can improve on the information lemma~\ref{lem:FullSoln.OneNonlocal.NondataDecays} provides about the behaviour of $\nu_0$.
    Firstly, integrating by parts on the whole interval $[0,1]$ instead of just $[1-\delta,1]$, we find the leading order term in $\Delta$ explicitly:
    \[
        \Delta(\la) = \frac{\ri}{\la}K(0) \re^{-\ri\la} + \bigoh\left(\la^{-2}\re^{-\ri\la}\right) \qquad \mbox{as } \la\to\infty \mbox{ within } \clos(D^+).
    \]
    We proceed from equation~\eqref{eqn:FullSoln.OneNonlocal.NondataDecays.proof2}.
    Integrating by parts, we find that, for $j=1,2$,
    \begin{multline*}
        \int_0^1 K(y) \re^{\ri\alpha^j\la y} \int_y^1 \re^{-\ri\alpha^j\la x} u(x,t') \D x \D y
        = \frac{1}{\ri\alpha^j\la} \left( - K(0) \int_0^1 \re^{-\ri \alpha^j\la x} u(x,t') \D x \right. \\
        + \left. \int_0^1 K(y) u(y,t') \D y
        - \int_0^1 \re^{\ri \alpha^j\la y} K'(y) \int_y^1 \re^{-\ri\alpha^j\la x} u(x,t') \D x \D y \right)
    \end{multline*}
    The first and third integrals on the right are both $\lindecayla$.
    The other is $h_2(t')$ by equation~\eqref{eqn:INVPwithBC.NC}.
    Therefore, equation~\eqref{eqn:FullSoln.OneNonlocal.NondataDecays.proof2} simplifies to
    \begin{equation} \label{eqn:FullSoln.OneNonlocal.NondataDecays.proof3}
        \nu_0(\la;t') = \frac{\re^{-\ri\la}}{\Delta(\la)} \left( \frac{2 h_2(t')}{\ri\la} - \int_0^1 K(y) \re^{\ri\la y} \int_0^y \re^{-\ri\la x} u(x,t') \D x \D y \right) + \lindecayla.
    \end{equation}
    To obtain the leading order behaviour of the remaining integral in equation~\eqref{eqn:FullSoln.OneNonlocal.NondataDecays.proof3}, we again integrate by parts:
    \begin{multline*}
        \int_0^1 K(y) \re^{\ri\la y} \int_0^y \re^{-\ri\la x} u(x,t') \D x \D y
        = \frac{1}{\ri\la} \left( K(1) \re^{\ri\la} \int_0^1 \re^{-\ri \la x} u(x,t') \D x - \int_0^1 K(y) u(y,t') \D y \right. \\
        - \left. \int_0^1 \re^{\ri \la y} K'(y) \int_0^y \re^{-\ri\la x} u(x,t') \D x \D y \right).
    \end{multline*}
    The first and third parenthetical terms are both $\lindecayla$ and the second term is known.
    Substituting into equation~\eqref{eqn:FullSoln.OneNonlocal.NondataDecays.proof3}, we find that
    \[
        \nu_0(\la;t') = \frac{\re^{-\ri\la}}{\Delta(\la)\ri\la} 3h_2(t') + \lindecayla = \frac{-3h_2(t')}{K(0)} + \lindecayla.
    \]
\end{rmk}

\begin{rmk}
    The D to N map arguments used, in conjunction with an evaluation of the global relation at $[0,1]$, an evaluation $z=1$ and integration over $y\in[0,1]$.
    This was arbitrary, in the sense that one could have instead employed an evaluation at $y=0$ and integration over $z\in[0,1]$.
    We emphasize that the three equations derived from the global relation thus form a system of rank $2$.
    However, problem~\eqref{eqn:INVPwithBC} admits a simpler D to N map when using our selection of evaluations, because it makes the D to N map separate into a system of three equations instead of the full six.
\end{rmk}

\begin{rmk}
    It is reasonable to attempt to extend the above presented arguments to a problem with no boundary conditions but three nonlocal conditions, each with a different weight $K_0,K_1,K_2$.
    Unfortunately, this is rather difficult, because the direct analogue of lemma~\ref{lem:FullSoln.OneNonlocal.NondataDecays}, appears to be false.
    Indeed, to obtain formulae for each of $f_0(\la;1),f_1(\la;1),f_2(\la;1)$, requires using versions of the nonlocal global relation~\eqref{eqn:nonlocalGR} with each weight $K_k$, introducing six new unknowns instead of two.
    Therefore, the $\la\mapsto\alpha^j\la$ maps are required for each nonlocal global relation, resulting in a full rank system of nine equations in nine unknowns.
    When solving this system via Cramer's rule, one finds exponentials $\re^{-2\ri\alpha^j\la}$ in the numerators, but only $\re^{-\ri\alpha^j\la}$ in the denominator, so the ratios are unbounded on the relevant sectors.
    It is an open question whether problems like this are illposed, wellposed but not amenable to a Fokas transform method approach, or open to anlaysis via an alternative version of the Fokas transfrom method.
    For example, it may be possible to adapt the method, outlined in~\cite[\S5.1--5.2]{MS2018a}, of understanding nonlocal value problems as weak-$\star$ limits of multipoint value problems, but then uniqueness of the solution must be proved by other means.
\end{rmk}

\begin{rmk}
    In~\cite{Smi2015a,FS2016a,PS2016a,ABS2022a,Smi2023a}, it was shown how to understand and even construct the Fokas transform method for IBVP, and in particular the objects analagous to $N_0,N_1$, in terms of the characteristic matrix of the classical (Lagrange) adjoint of the spatial differential operator.
    Because the spatial differential we study in this paper has a nonlocal condition, it does not have a classical adjoint, so it is not clear what the analagous construction should be.
    More examples like this one, or an analysis via weak-$\star$ limits of multipoint operators (which do have classical adjoints~\cite{Loc1973a,ABS2022a}) may be illustrative.
\end{rmk}

\begin{rmk}
    Because the Airy equation is third order, it behaves substantially differently if time (equivalently, space) is run in the other direction.
    Indeed, it is expected that simply replacing equation~\eqref{eqn:INVPwithBC.PDE} with $[\partial_t-\partial_x^3]u(x,t)=0$ would make INVP~\eqref{eqn:INVPwithBC} illposed.
    It is known that separated IBVP for the Airy equation require exactly one boundary condition specified at the left and two at the right in order to be wellposed, and the alternative must have two at the left and one at the right~\cite{Pel2004a}.
    In this context, the requirement in theorem~\ref{thm:FullSoln.OneNonlocal.Solution.D} that $K$ be nonzero at $0$ is not simply a technical imposition inherited from its lemmata but a fundamental requirement for wellposedness.
    For an INVP with the alternative PDE to be wellposed, it is expected that not only would the $K$ nonzero at $0$ requirement remain, but one of the boundary conditions at $x=1$ would have to be substituted for a further boundary condition at $x=0$.
    A full characterization of wellposedness for third order IBVP including those with boundary conditions coupling between the two ends has still not been obtained, despite indications that it is related to the criteria for Birkhoff regularity~\cite{Loc2000a,Smi2012a,ABS2022a}.
    Therefore, we relegate also to later work an investigation of how wellposedness of INVP for the Airy equation is affected by selection of different boundary and nonlocal conditions.
\end{rmk}

\begin{rmk}
    This problem can be understood as related to that of a linearization of the physical problem of unidirectional waves in shallow water, where the surface elevation is difficult to measure at a point $x=0$ but relatively easy to measure on average over an interval, using a measuring device with sensitivity $K(x)$.
    This measurement may be related to that obtained from a pressure plate.
    In the small amplitude linearization of the Korteweg de Vries equation, equation~\eqref{eqn:INVPwithBC.PDE} would also have a $u_x$ term.
    Including that term would make the formulae above more complicated but, guided by the results for IBVP, we do not expect its inclusion would significantly change the character of the results.
    
    Related problems have $K$ the sum of a function of bounded variation and a one sided delta distribution or derivative delta distribution at $x=0$.
    This effectively transforms the nonlocal condition into a hybrid boundary nonlocal condition.
    Such problems have applications in the feedback stabilizability and boundary controllability of the system via backstepping; see~\cite{KS2008a} for a general survey and, including the additional $u_x$ term in the PDE,~\cite{OB2019a}.
    We expect that such problems may be studied via the means presented above, but the analysis in the proofs of results analogous to lemmata~\ref{lem:FullSoln.OneNonlocal.ZerosLocus} and~\ref{lem:FullSoln.OneNonlocal.NondataDecays} would be slightly simplified because the presence of the $\partial_x^ku(0,t)$ term in the hybrid nonlocal boundary condition can provide an extra $\lindecayla$ decay in some of the numerator integrals.
\end{rmk}

\begin{rmk}
    Having presented the full unified transform method for INVP~\eqref{eqn:INVPwithBC}, some discussion is warranted of how this relates to the method for the heat equation presented in~\cite{MS2018a}.
    
    To emphasize the ready adaptability of the method, we have been careful to present the broader argument in section~\ref{sec:FullSoln} in a manner closely paralleling the heat equation paper.
    The arguments to justify lemmata~\ref{lem:FullSoln.OneNonlocal.ZerosLocus} and~\ref{lem:FullSoln.OneNonlocal.NondataDecays} require some more careful bounds for the third order problem.
    This arises from the geometric complications of the exponential sum
    $\sum_{k=0}^1\sum_{j=0}^2\exp((-1)^k\ri\alpha^j\la)$
    having 6 terms which may be dominant as $\la\to\infty$ in various sectors, while the corresponding
    $\sum_{k=0}^1\sum_{j=0}^1\exp((-1)^k\ri\alpha_2^j\la)$
    (in which $\alpha_2=-1$ is a primitive square root of unity)
    arising for the heat equation has only two.
    But this point of contrast is one of detail rather than representing a fundamental divergence of the argument.
    
    The most striking point is one of comparison rather than contrast.
    In lemma~\ref{lem:FullSoln.OneNonlocal.ZerosLocus}, the first asymptotic statement is of $\bigoh(1)$ behaviour rather than $\lindecayla$, and this is reflected in equation~(2.3) of~\cite{MS2018a}.
    These two asymptotic bounds are very different to the $\lindecayla$ bounds one typically obtains in all relevant sectors in the equivalent lemmata for IBVP.
    They necessitate the more careful application of Jordan's lemma in which the $\exp(\ri\la x)$ factor is split, which works if and only if the relevant sector boundaries are nonparallel with the real line.
    It remains to be investigated whether this is a fundamental feature of the unified transform method for INVP, and how it might affect wellposedness for problems with different boundary and nonlocal problems.
\end{rmk}

\section{Time periodic problems for the Airy equation} \label{sec:Periodic}

In this section, we study the time periodic analogue of problem~\eqref{eqn:INVPwithBC}.
We modify the arguments of~\cite{FvdW2021a,FPS2022a} for the present setting, informed by the nonlocal global relation concepts that were introduced in~\cite{MS2018a} and adapted to the Airy equation above.

\subsection{Time periodic problem} \label{ssec:Periodic.Problems}

For $T>0$, we say that function $d:[0,\infty)\to \CC$ is \emph{$T$ periodic} if it can be represented via its exponential Fourier series
\[
    d(t) = \sum_{n\in\ZZ} D_n \re^{\ri n\omega t}, \qquad
    \omega=\frac {2\pi}T,
\]
almost everywhere.
Occasionally, we tacitly extend the domain of $d$ to $\RR$ so that it is periodic in the usual sense.
This is not intended to suggest that all of the problems we study have solutions for negative time; it is merely a notational convenience.

We shall study the following problem, which is the time periodic analogue of problem~\eqref{eqn:INVPwithBC}, because the initial condition has been replaced with a time periodicity assumption.
\begin{prob*}[Time periodic finite interval problem for the Airy equation with one nonlocal condition]
    \begin{subequations} \label{eqn:PerNVPwithBC}
    \begin{align}
        \label{eqn:PerNVPwithBC.PDE} \tag{\theparentequation.PDE}
        [\partial_t+\partial_x^3] q(x,t) &= 0 & (x,t) &\in (0,1) \times (0,\infty), \\
        \label{eqn:PerNVPwithBC.IC} \tag{\theparentequation.Per}
        q(x,t) &= q(x,t+T) & (x,t) &\in[0,1] \times [0,\infty), \\
        \label{eqn:PerNVPwithBC.BC0} \tag{\theparentequation.BC0}
        q(1,t) &= h_0(t) & t &\in[0,\infty), \\
        \label{eqn:PerNVPwithBC.BC1} \tag{\theparentequation.BC1}
        q_x(1,t) &= h_1(t) & t &\in[0,\infty), \\
        \label{eqn:PerNVPwithBC.NC} \tag{\theparentequation.NC}
        \int_0^1 K(y)q(y,t)\D y &= h_2(t) & t &\in[0,\infty),
    \end{align}
    in which $h_k$, and $K$ are sufficiently smooth, $h_k(t)=h_k(t+T)$ for all $t\in[0,\infty)$, and $T>0$ is fixed.
    \end{subequations}
\end{prob*}

In~\cite[proposition~2]{FPS2022a}, the asymptotically valid D to N map was derived for an asymptotically time periodic problem similar to~\eqref{eqn:INVPwithBC}.
The differences being that the PDE was allowed to have any linear spatial differential operator, but only boundary conditions were admitted, not nonlocal conditions.
Here, we specialise to the Airy equation, but admit the more complicated nonlocal condition~\eqref{eqn:INVPwithBC.NC}.
More significantly, we streamline the argument by focussing directly on the problem for $q$, which is assumed to be truly periodic, not just asymptotically so.

\subsection{The periodic relations} \label{ssec:Periodic.Relations}
Because $q$ satisfies the same PDE as did $u$ in~\S\ref{sec:FullSoln}, $q$ also satisfies equation~\eqref{eqn:PreGR}, provided $0\leq y \leq z \leq 1$.
When evaluated at $y=0$, $z=1$, this equation is named ``the $Q$ equation'' in~\cite{FvdW2021a}, after a change of notation $Q(\la) = -\hat{q}(\la;t)$.
To emphasize the simpler derivation of this equation presented herein (and to preserve the symbol $Q$ for later use), we do not adopt this notation.
To extend the arguments of~\cite{FPS2022a} to nonlocal problems, we shall employ both the $[0,1]$ and the $\int[y,1]$ versions of equation~\eqref{eqn:PreGR}:
\begin{align}
    \notag
    \left[ \partial_t - \ri\la^3 \right] \hat{q}(\la;t)
    &=
    \left[ \partial_{xx} q(0,t) + \ri\la\partial_xq(0,t) - \la^2q(0,t) \right] \\
    \label{eqn:PrePeriodicRelation}
    &\hspace{4em} - \re^{-\ri\la} \left[ \partial_{xx} q(1,t) + \ri\la\partial_xq(1,t) - \la^2q(1,t) \right], \\
    \notag
    \left[ \partial_t - \ri\la^3 \right] \int_0^1 \re^{\ri\la y} \hat{q}(\la;t;y,1) K(y) \D y
    &=
    \int_0^1 \left[ \partial_{xx} q(y,t) + \ri\la\partial_xq(y,t) - \la^2q(y,t) \right] K(y) \D y \\
    \label{eqn:PreNonlocalPeriodicRelation}
    &\hspace{-3.2em} - \int_0^1 \re^{-\ri\la(1-y)} K(y) \D y \left[ \partial_{xx} q(1,t) + \ri\la\partial_xq(1,t) - \la^2q(1,t) \right],
\end{align}
mirroring the approach of~\S\ref{sec:FullSoln}, wherein both $[0,1]$ and $\int[y,1]$ versions of the global relation~\eqref{eqn:GR} were used.

Because $q$ is assumed time periodic, so also are its spatial Fourier transform and its boundary values.
Hence, we may express all terms in the above equations using their exponential Fourier series over the interval $[0,T]$.
\begin{equation} \label{eqn:Periodic.Notation}
\begin{gathered}
    \sum_{n\in\ZZ} G_n^{j} \re^{\ri n \omega t} = g^{j}(t) = \partial_x^j q(0,t), \qquad\qquad
    \sum_{n\in\ZZ} \Gamma_n^{j} \re^{\ri n \omega t} = \gamma^{j}(t) = \partial_x^j q(1,t), \\
    \sum_{n\in\ZZ} H_n^{k} \re^{\ri n \omega t} = h_k(t), \qquad\qquad
    \sum_{n\in\ZZ} Q_n(\la) \re^{\ri n \omega t} = \hat{q}(\la;t), \\
    \sum_{n\in\ZZ} P_n(\la) \re^{\ri n \omega t} = p(\la;t) = \int_0^1 \re^{\ri\la y} K(y)\hat{q}(\la;t,y,1) \D y, \\
    \sum_{n\in\ZZ} A_n^j \re^{\ri n \omega t} = \Msups{a}{}{}{j}{k}(t) = \int_0^1 K(y) \partial_x^j q(y,t) \D y,
\end{gathered}
\end{equation}
where each Fourier coefficient may be calculated as
\[
    \Phi_n = \frac1T\int_0^T \phi(t) \re^{\ri n \omega t} \D t.
\]
In the above notation, $\omega=2\pi/T$ is the angular frequency, $g$ are the left boundary values, $\gamma$ the right boundary values, $h$ the boundary data, and $a$ are nonlocal values.
In problem~\eqref{eqn:PerNVPwithBC}, $a^0$ is specified via nonlocal condition~\eqref{eqn:PerNVPwithBC.NC}, while $a^1$, $a^2$ are unknown.
By the Fourier basis property, equality of Fourier series is equivalent to equality of corresponding Fourier coefficients.
Therefore, equations~\eqref{eqn:PrePeriodicRelation} and~\eqref{eqn:PreNonlocalPeriodicRelation} simplify to, for all $n\in\ZZ$ and all $\la\in\CC$,
\begin{align}
    \label{eqn:PeriodicRelation}
    (\ri n\omega - \ri\la^3) Q_n(\la) &= \left[G_n^2 + \ri\la G_n^1 - \la^2 G_n^0\right] - \re^{-\ri\la} \left[\Gamma_n^2 + \ri\la \Gamma_n^1 - \la^2 \Gamma_n^0\right], \\
    \label{eqn:NonlocalPeriodicRelation}
    (\ri n\omega - \ri\la^3) P_n(\la)
    &= \left[A_n^2 + \ri\la A_n^1 - \la^2 A_n^0\right]
    - \re^{-\ri\la}\widehat{K}(-\la) \left[\Gamma_n^2 + \ri\la \Gamma_n^1 - \la^2 \Gamma_n^0\right].
\end{align}
We refer to these equations as the \emph{periodic relation} and \emph{nonlocal periodic relation}, respectively.
Other than the notational changes, in the derivation of these relations from equation~\eqref{eqn:PreGR}, the only step was to calculate exponential Fourier coefficients.

\begin{rmk}
    To aid solution of the INVP in~\S\ref{sec:FullSoln}, from equation~\eqref{eqn:PreGR} was derived the global relation on $[0,1]$~\eqref{eqn:GR} by solution of an ODE in time.
    The nonlocal global relation~\eqref{eqn:nonlocalGR} followed by taking the same integrals against $K$.
    Thus, the (nonlocal) periodic relation is the analogue of the (nonlocal) global relation, but derived via temporal Fourier series expansion instead of solution of a temporal ODE.
    This is a reasonable approach because, lacking an initial condition in problem~\eqref{eqn:PerNVPwithBC}, it is impossible to solve the temporal ODE, but periodicity means that temporal Fourier expansion is possible.
    Moreover, the fact that the global relations could be used to construct D to N maps for the INVP suggests the periodic relations may be valuable in finding D to N maps for the periodic nonlocal value problems.
\end{rmk}

\subsection{Periodic D to N map} \label{ssec:Periodic.DtoN}
We construct the D to N map from the periodic relation~\eqref{eqn:PeriodicRelation} and the nonlocal periodic relation~\eqref{eqn:NonlocalPeriodicRelation}.
Along with the boundary values $G_n^k$ and $\Gamma_n^k$, these equations feature, within $Q_n$ and $P_n$, the unknown spatial Fourier transforms of $q$.
Construction of the periodic D to N map requires elimination of the terms $P_n$ and $Q_n$.
In our study of INVP in~\S\ref{sec:FullSoln}, $\hat{u}(\la;t)$ was treated as if it were data and shown only at the conclusion of the argument not to contribute to the derived solution formula.
Fortunately, for the periodic nonlocal value problems, a simpler approach is possible.

Because the coefficient $(\ri n\omega - \ri\la^3)$ of $Q_n$ and $P_n$ in equations~\eqref{eqn:PeriodicRelation}--\eqref{eqn:NonlocalPeriodicRelation} has zeros,
$Q_n$ and $P_n$ may be eliminated by selecting the particular $\la$ values
\begin{equation} \label{eqn:defn.la_n}
    \la=\alpha^j \la_n,
    \qquad\mbox{where}\qquad
    \la_n = \sgn(n) \sqrt[3]{\abs{n}\omega}, \qquad j\in\{0,1,2\}, \qquad n\in\ZZ,
\end{equation}
and $\alpha=\exp(2\pi\ri/3)$. The finitude of $P_n(\la)$ and $Q_n(\la)$ for all $\la\in\CC$ follows from the fact that they are defined as compact integrals of $q$, which is assumed to be integrable in both space and time.
Thus, the left sides of the (nonlocal) periodic relations vanish, leaving equations linking only the boundary and nonlocal values.
We shall solve the resulting system.

For notational convenience, we define $\kappa(y)=K(1-y)$ so that $\hat\kappa(\la)=\re^{-\ri\la}\widehat{K}(-\la)$.
The nonlocal periodic relation~\eqref{eqn:NonlocalPeriodicRelation} reduces to
\begin{equation} \label{eqn:OneNonlocal.NonlocalPeriodicRelation.Reduced}
    (\ri n\omega - \ri\la^3) P_n(\la) =
    \left[\Msup{A}{n}{}{2} + \ri\la \Msup{A}{n}{}{1} - \la^2 H_n^2\right]
    - \hat\kappa(\la) \left[\Gamma_n^2 + \ri\la H_n^1 - \la^2 H_n^0\right].
\end{equation}
For each nonzero integer $n$, the left side of this equation has three zeros, as given by equation~\eqref{eqn:defn.la_n}.
Therefore,
\begin{equation} \label{eqn:OneNonlocal.NonlocalPeriodicRelation.System}
    \begin{pmatrix}
        1 & \ri\la_n & -\hat\kappa(\la_n) \\
        1 & \ri\alpha\la_n & -\hat\kappa(\alpha\la_n) \\
        1 & \ri\alpha^2\la_n & -\hat\kappa(\alpha^2\la_n)
    \end{pmatrix}
    \begin{pmatrix}
        \Msup{A}{n}{}{2} \\ \Msup{A}{n}{}{1} \\ \Gamma_n^2
    \end{pmatrix}
    =
    \begin{pmatrix}
        \mathcal{N}_n(\la_n) \\
        \mathcal{N}_n(\alpha\la_n) \\
        \mathcal{N}_n(\alpha^2\la_n)
    \end{pmatrix},
\end{equation}
where
\[
    \mathcal{N}_n(\la) = \la^2 H_n^2 + \ri\la\hat\kappa(\la) H_n^1 - \la^2\hat\kappa(\la) H_n^0.
\]
If instead $n=0$, then system~\eqref{eqn:OneNonlocal.NonlocalPeriodicRelation.System} is guaranteed to be rank deficient.
However, differentiating equation~\eqref{eqn:OneNonlocal.NonlocalPeriodicRelation.Reduced} twice with respect to $\la$ and evaluating at $\la=0$ yields an equation for $\Gamma_0^2$.
Therefore,
\begin{equation} \label{eqn:OneNonlocal.NonlocalPeriodicRelation.Gamma.n2}
    \Gamma_n^2 =
    \begin{cases}
        2\left( \frac{\D}{\D\la}\hat\kappa(0) H_0^1 - \hat\kappa(0) H_0^0 + H_0^2 \right) / \frac{\D^2}{\D\la^2}\hat\kappa(0) & \mbox{if } n = 0, \\
        \sum_{j=0}^2 \alpha^j \mathcal{N}_n(\alpha^j\la_n) / \sum_{j=0}^2 \alpha^j \hat\kappa(\alpha^j\la_n), & \mbox{otherwise.}
    \end{cases}
\end{equation}
The exceptional case is if either of the denominators in equation~\eqref{eqn:OneNonlocal.NonlocalPeriodicRelation.Gamma.n2} is zero.
In that case, for generic $h_k$, no solution for $\Gamma_n^2$ exists.
Therefore, we have found a necessary condition for periodic nonlocal value problem~\eqref{eqn:PerNVPwithBC} to have a solution:
\begin{crit} \label{crit:Periodic.OneNonlocal.Solubility}
    Both $\frac{\D^2}{\D\la^2}\hat\kappa(0)$ and, for all $n\in\ZZ\setminus\{0\}$, $\sum_{j=0}^2 \alpha^j \hat\kappa(\alpha^j\la_n)$ are nonzero.
\end{crit}

We proceed under our original assumption that $q$ exists, which implies that equation~\eqref{eqn:OneNonlocal.NonlocalPeriodicRelation.Gamma.n2} is valid.
Evaluating the periodic relation~\eqref{eqn:PeriodicRelation} at the $\la$ values~\eqref{eqn:defn.la_n}, we determine that
\begin{equation} \label{eqn:OneNonlocal.PeriodicRelation.System}
    \begin{pmatrix}
        1 & \ri\la_n & -\la_n^2 \\
        1 & \ri\alpha\la_n & -\alpha^2\la_n^2 \\
        1 & \ri\alpha^2\la_n & -\alpha\la_n^2
    \end{pmatrix}
    \begin{pmatrix}
        G_n^2 \\ G_n^1 \\ G_n^0
    \end{pmatrix}
    =
    \begin{pmatrix}
        \mathfrak{N}_n(\la) \\
        \mathfrak{N}_n(\alpha\la) \\
        \mathfrak{N}_n(\alpha^2\la)
    \end{pmatrix},
\end{equation}
where
\[
    \mathfrak{N}_n(\la) = \re^{-\ri\la} \left[\Gamma_n^2 + \ri\la H_n^1 - \la^2 H_n^0\right].
\]
System~\eqref{eqn:OneNonlocal.PeriodicRelation.System} has determinant $-3\sqrt3\la_n^3$, so is full rank for all $n\neq0$.
To determine the zero indexed left boundary Fourier coefficients, we use equations obtained by differentiating the periodic relation~\eqref{eqn:PeriodicRelation} 0, 1, and 2 times, respectively, and evaluating at $\la=0$.
The solution is
\begin{subequations} \label{eqn:Periodic.OneNonlocal.G}
\begin{align}
    G_n^2 &=
    \begin{cases}
        \Gamma_0^2 & \mbox{if } n=0, \\
        \frac13\sum_{j=0}^2\mathfrak{N}_n(\alpha^j\la_n) & \mbox{otherwise,}
    \end{cases} \\
    G_n^1 &=
    \begin{cases}
        H_0^1 - \Gamma_0^2 & \mbox{if } n=0, \\
        \frac1{3\ri\la_n}\sum_{j=0}^2\alpha^{2j}\mathfrak{N}_n(\alpha^j\la_n) & \mbox{otherwise,}
    \end{cases} \\
    G_n^0 &=
    \begin{cases}
        H_0^0 - H_0^1 + \frac12 \Gamma_0^2 & \mbox{if } n=0, \\
        \frac1{-3\la_n^2}\sum_{j=0}^2\alpha^{j}\mathfrak{N}_n(\alpha^j\la_n) & \mbox{otherwise.}
    \end{cases}
\end{align}
\end{subequations}

\subsection{Solution of the time periodic problem} \label{ssec:Periodic.Solution}

\begin{prop} \label{prop:Periodic.Solution.OneNonlocal}
    Suppose $q$ satisfies problem~\eqref{eqn:PerNVPwithBC} with generic data piecewise continuous $h_k$, and $K$ is such that criterion~\ref{crit:Periodic.OneNonlocal.Solubility} holds for $\kappa(y)=K(1-y)$.
    Then
    \[
        q(x,t) = \frac1{2\pi} \int_{-\infty}^\infty \re^{\ri\la x} \sum_{n\in\ZZ} \re^{\ri n \omega t} \frac{\left[ G_n^2 + \ri\la G_n^1 - \la^2 G_n^0 \right] - \re^{-\ri\la}\left[ \Gamma_n^2 + \ri\la H_n^1 - \la^2 H_n^0 \right]}{\ri n \omega - \ri\la^3} \D\la,
    \]
    where $H_n^j$ are defined in equations~\eqref{eqn:Periodic.Notation}, $\Gamma_n^2$ are defined in equation~\eqref{eqn:OneNonlocal.NonlocalPeriodicRelation.Gamma.n2}, and $G_n^j$ are defined in equations~\eqref{eqn:Periodic.OneNonlocal.G}.
\end{prop}

\begin{proof}
    The Fourier coefficients on the right of equation~\eqref{eqn:PeriodicRelation} are determined and the criterion justified above.
    Formally, the result follows from equation~\eqref{eqn:PeriodicRelation} by the fourth of equations~\eqref{eqn:Periodic.Notation} and the usual inversion theorem for the spatial Fourier transform.
    It remains only to justify the convergence of the series and integral.
    
    By Parseval's theorem, for each $k$, $(H_n^k)_{n\in\ZZ}$ is square summable.
    By definition, $\Gamma_n^2$ has dominant term $\lambda_n^2 H_n^0$, and $\lambda_n = \Theta(n^{1/3})$, so $(\Gamma_n^2/n)_{n\in\ZZ}$ is also square summable.
    Similarly, for each $j$, $(G_n^j/n)_{n\in\ZZ}$ are also square summable.
    Because $G_0^j$ and $\Gamma_0^2$ were constructed to ensure that the numerator has a zero of order 3 at $0$, and because the powers of $\la$ appearing in the numerator are each no greater than $2$, the convergence is uniform in $\la$, and the value of the series decays as $\la\to\pm\infty$.
    Hence, by the Riemann-Lebesgue lemma, the integral also converges.
\end{proof}

\begin{rmk}
    We have not investigated the failure of criterion~\ref{crit:Periodic.OneNonlocal.Solubility}.
    It may be that this case corresponds, as it does in certain IBVP for the linear Schr\"odinger equation~\cite{FPS2022a,Duj2009a}, to INVP whose solutions blow up linearly in time, despite having periodic boundary data.
    We emphasize that it is unsurprising that illposedness of periodic problems like problem~\eqref{eqn:PerNVPwithBC} can occur.
    For example, in the heat equation, an IBVP with a nonzero Neumann conditions should be expected have a solution that blows up in time, rather than behaving periodically.
    
    To investigate more precisely the illposedness would require a careful asymptotic analysis of formula~\eqref{eqn:FullSoln.DtoN.OneNonlocal.Solution.D}; the method of~\S\ref{sec:Periodic} cannot provide any insight on the cause of its own failure, because it is predicated on the existence of a solution to problem~\eqref{eqn:PerNVPwithBC}.
\end{rmk}

\section{Long time asymptotics} \label{sec:LongTime}

In this section, we study the long time behaviour of solutions to problem~\eqref{eqn:INVPwithBC} with either zero or periodic boundary and nonlocal data.
The results are stated in proposition~\ref{prop:LongTime.HomogeneousINVP.LinDec} and theorem~\ref{thm:LongTime.MainResult}.

We begin by studying the following problem, which is a special case of problem~\eqref{eqn:INVPwithBC} with homogeneous boundary and nonlocal conditions.
\begin{prob*}[Homogeneous finite interval problem for the Airy equation with one nonlocal condition]
    \begin{subequations} \label{eqn:HomogINVPwithBC}
    \begin{align}
        \label{eqn:HomogINVPwithBC.PDE} \tag{\theparentequation.PDE}
        [\partial_t+\partial_x^3] v(x,t) &= 0 & (x,t) &\in (0,1) \times (0,\infty), \\
        \label{eqn:HomogINVPwithBC.IC} \tag{\theparentequation.IC}
        v(x,0) &= V(x) & x &\in[0,1], \\
        \label{eqn:HomogINVPwithBC.BC0} \tag{\theparentequation.BC0}
        v(1,t) &= 0 & t &\in[0,\infty), \\
        \label{eqn:HomogINVPwithBC.BC1} \tag{\theparentequation.BC1}
        v_x(1,t) &= 0 & t &\in[0,\infty), \\
        \label{eqn:HomogINVPwithBC.NC} \tag{\theparentequation.NC}
        \int_0^1 K(y)v(y,t)\D y &= 0 & t &\in[0,\infty),
    \end{align}
    in which $V$ and $K$ are sufficiently smooth.
    \end{subequations}
\end{prob*}
This is the same as problem~\eqref{eqn:INVPwithBC}, except that here we insist $h_0=h_1=h_2=0$.
In particular, corollary~\ref{cor:FullSoln.OneNonlocal.Solution.E} applies, with simplifications of the definitions of $N_0$ and $N_1$.
An asymptotic analysis of this formula for $v(x,t)$, presented at the end of this section allows us to prove the following proposition.

\begin{prop} \label{prop:LongTime.HomogeneousINVP.LinDec}
    Suppose $v$ satisfies problem~\eqref{eqn:HomogINVPwithBC} and the criteria of theorem~\ref{thm:FullSoln.OneNonlocal.Solution.D} hold.
    Then, uniformly in $x$, $v(x,t) = \bigoh(t^{-1})$ as $t\to\infty$.
\end{prop}

Now consider problem~\eqref{eqn:INVPwithBC}, but with all of $h_0,h_1,h_2$ having common period $T$.
We seek an expression for $u(x,t)$ valid in the long time regime.
We use the principle of linear superposition to express $u(x,t) = q(x,t) + v(x,t)$, with $q$ satisfying problem~\eqref{eqn:PerNVPwithBC} and $v$ satisfying problem~\eqref{eqn:HomogINVPwithBC} in which $V(x) = U(x)-q(x,0)$.
By proposition~\ref{prop:LongTime.HomogeneousINVP.LinDec}, $v$ decays.
From proposition~\ref{prop:Periodic.Solution.OneNonlocal}, we can obtain an expression for $q$, hence an asymptotically valid expression for $u$.
This justifies the following theorem.

\begin{thm} \label{thm:LongTime.MainResult}
    Suppose $u$ satisfies problem~\eqref{eqn:INVPwithBC} in which all of $h_0,h_1,h_2$ are generic piecewise continuous functions with common period $T$.
    Suppose that the corresponding homogeneous problem~\eqref{eqn:HomogINVPwithBC} satisfies the criteria of theorem~\ref{thm:FullSoln.OneNonlocal.Solution.D}.
    Suppose that $K$ is such that $\kappa(y)=K(1-y)$ satisfies criterion~\ref{crit:Periodic.OneNonlocal.Solubility}.
    Then
    \[
        u(x,t) = \frac1{2\pi} \int_{-\infty}^\infty \re^{\ri\la x} \sum_{n\in\ZZ} \re^{\ri n \omega t} \frac{\left[ G_n^2 + \ri\la G_n^1 - \la^2 G_n^0 \right] - \re^{-\ri\la}\left[ \Gamma_n^2 + \ri\la H_n^1 - \la^2 H_n^0 \right]}{\ri n \omega - \ri\la^3} \D\la + \bigoh(t^{-1}),
    \]
    where $H_n^j$ are defined in equations~\eqref{eqn:Periodic.Notation}, $\Gamma_n^2$ are defined in equation~\eqref{eqn:OneNonlocal.NonlocalPeriodicRelation.Gamma.n2}, and $G_n^j$ are defined in equations~\eqref{eqn:Periodic.OneNonlocal.G}.
\end{thm}

\begin{proof}[Proof of proposition~\ref{prop:LongTime.HomogeneousINVP.LinDec}]
    Let
    \[
        n(\la) = \frac1{\Delta(\la)} \sum_{j=0}^2 \alpha^jN_1(\la)(\alpha^j\la),
    \]
    and observe that, because the boundary and nonlocal conditions are homogeneous,
    \[
        N_1(\la) = \int_0^1K(y)\re^{\ri\la y}\widehat{V}(\la;y,1) \D y;
    \]
    both $N_1$ and $n$ are truly independent of $t'$.
    By corollary~\ref{cor:FullSoln.OneNonlocal.Solution.E}, $2\pi u(x,t) = I^+(t)+I^-(t)$, where we have suppressed the $x$ depenence of $I^\pm$, and
    \begin{align*}
        I^+(t) &=
        \int_{\partial E_R^+} \re^{\ri\la x+\ri\la^3 t} \left( \re^{-\ri\la}n(\la) - \widehat{V}(\la) \right) \D\la, \\
        I^-(t) &=
        \int_{\partial E_R^-} \re^{\ri\la (x-1)+\ri\la^3 t} n(\la) \D\la.
    \end{align*}
    We will analyse each of $I^\pm$ separately.
    
    Interpreting, as we must due to its derivation from equation~\eqref{eqn:PreEhrenpreisForm}, $I^-(t)$ as a Cauchy principal value contour integral, and integrating by parts,
    \[
        I^-(t) =
        \frac1{3\ri t} \lim_{b\to+\infty} \left\{
        \left[ \re^{\ri\la^3t}\frac{\re^{\ri\la(x-1)}n(\la)}{\la^2} \right]^{\la=\alpha^{2}b}_{\la=\alpha^{-1/2}b}
        - \int_{\Gamma_b^-} \re^{\ri\la^3t} \frac{\D}{\D\la} \left[ \frac{\re^{\ri\la(x-1)}n(\la)}{\la^2} \right] \D\la
        \right\},
    \]
    where $\Gamma_b^-=\partial E_R^- \cap B(0,b)$.
    Because $\alpha^2b,\alpha^{-1/2}b\in\clos(D_R^-)$, lemma~\ref{lem:FullSoln.OneNonlocal.NondataDecays} guarantees that $n(\la)=\bigoh(b^{-1})$.
    Along the rays $\alpha^2\RR^+$ and $\alpha^{-1/2}\RR^+$, $\abs{\re^{\ri\la^3t}}=1$ and $\abs{\re^{\ri\la(x-1)}}\leq1$, so the decay of both boundary terms as $b\to+\infty$ is justified:
    \begin{equation} \label{eqn:LongTime.HomogeneousINVP.LinDec.Proof.1}
        I^-(t) =
        \frac1{3\ri t} \lim_{b\to+\infty} \left\{
        0
        - \int_{\Gamma_b^-} \re^{\ri\la^3t} \frac{\D}{\D\la} \left[ \frac{\re^{\ri\la(x-1)}n(\la)}{\la^2} \right] \D\la
        \right\}.
    \end{equation}
    Differentiating,
    \[
        \re^{\ri\la^3t} \frac{\D}{\D\la} \left[ \frac{\re^{\ri\la(x-1)}n(\la)}{\la^2} \right]
        =
        \frac1{\la^2}
        \times
        \re^{\ri\la^3t+\ri\la(x-1)}
        \times
        \left[ \left(\ri(x-1)-\frac2\la\right) n(\la) + n'(\la) \right].
    \]
    The exponential factor is, uniformly in $x,t$, bounded in modulus by $1$, and $n(\la)$ decays as $\la\to\infty$ along $\partial E_R^\pm$.
    We shall argue that $n'(\la)$ is also bounded (in fact decaying), whence dominated convergence gives a bound on the $b\to+\infty$ limit in equation~\eqref{eqn:LongTime.HomogeneousINVP.LinDec.Proof.1} that is uniform in both $x$ and $t$.
    We calculate
    \[
        n'(\la) = \frac{1}{\Delta(\la)} \sum_{j=0}^2 \alpha^{2j} N'_1(\alpha^j\la) - \frac{\Delta'(\la)}{[\Delta(\la)]^2} \sum_{j=0}^2 \alpha^{2j} N_1(\alpha^j\la),
    \]
    where
    \[
        \Delta'(\la) = -\ri \sum_{j=0}^2 \alpha^{2j} \int_0^1 \re^{-\ri\la\alpha^jy}y\kappa(y)\D y = \bigoh(\Delta(\la)),
    \]
    and
    \[
        N'_1(\alpha^j\la) = -\ri\alpha^j \int_0^1 K(y)\re^{\ri\alpha^j\la y} \int_y^1 \re^{-\ri\alpha^j\la z} (y-z) V(z) \D z \D y = \bigoh\left(\la^{-1}\Delta(\la)\right),
    \]
    as $\la\to\infty$ along the contour $\partial E_R^\pm$, hence $n'(\la)$ also decays.
    This shows that $I^-(t)=\bigoh(t^{-1})$, uniformly in $x$, as $t\to\infty$.
    
    The argument for $I^+$ follows the same structure as that for $I^-$.
    Integrating by parts,
    \begin{multline*}
        I^+(t) =
        \frac1{3\ri t} \lim_{b\to+\infty} \left\{
        \left[ \re^{\ri\la^3t+\ri\la x}\frac{\re^{-\ri\la}n(\la)-\widehat{V}(\la)}{\la^2} \right]^{\la=-b,\,\alpha^{1/2}b}_{\la=\alpha b,\,b} \right. \\
        \left. - \int_{\Gamma_b^+} \re^{\ri\la^3t} \frac{\D}{\D\la} \left[ \re^{\ri\la x} \frac{\re^{-\ri\la}n(\la)-\widehat{V}(\la)}{\la^2} \right] \D\la
        \vphantom{\left[ \re^{\ri\la^3t+\ri\la x}\frac{\re^{-\ri\la}n(\la)-\widehat{V}(\la)}{\la^2} \right]^{\la=-b,\,\alpha^{1/2}b}_{\la=\alpha b,\,b}}
        \right\},
    \end{multline*}
    where $\Gamma_b^+=\partial E_R^+ \cap B(0,b)$.
    By lemma~\ref{lem:FullSoln.OneNonlocal.NondataDecays},
    \[
        \re^{-\ri\la}n(\la)-\widehat{V}(\la) = \bigoh\left(1\right),
    \]
    as $\la\to\infty$ along $\partial D_R^+$, and the same argument applies to the limit as $\la\to\infty$ along $\RR$.
    Hence, this numerator is bounded on $\partial E_R^+$,
    and the boundary terms decay in the limit $b\to+\infty$;
    \begin{equation}
        \label{eqn:LongTime.HomogeneousINVP.LinDec.Proof.2}
        I^+(t) =
        \frac1{3\ri t} \lim_{b\to+\infty} \left\{
        0
        - \int_{\Gamma_b^+} \re^{\ri\la^3t} \frac{\D}{\D\la} \left[ \re^{\ri\la x} \frac{\re^{-\ri\la}n(\la)-\widehat{V}(\la)}{\la^2} \right] \D\la
        \vphantom{\left[ \re^{\ri\la^3t+\ri\la x}\frac{\re^{-\ri\la}n(\la)-\widehat{V}(\la)}{\la^2} \right]^{\la=-b,\,\alpha^{1/2}b}_{\la=\alpha b,\,b}}
        \right\}.
    \end{equation}
    
    The remaining integrand can be reexpressed as
    \begin{equation}
        \label{eqn:LongTime.HomogeneousINVP.LinDec.Proof.3}
        \frac{1}{\la^2} \times \re^{\ri\la^3t+\ri\la x} \times \left[ \left(\ri x-\frac{2}{\la}\right)\left( \re^{-\ri\la}n(\la)-\widehat{V}(\la) \right) + \frac{\D}{\D\la}\left( \re^{-\ri\la}n(\la)-\widehat{V}(\la) \right) \right],
    \end{equation}
    and the first term in the bracket, as above, is $\bigoh(1)$.
    To obtain a bound on the derivative, we apply a similar reexpression to that used in the proof of lemma~\ref{lem:FullSoln.OneNonlocal.NondataDecays}, but keep the lower order terms.
    Because
    \[
        \frac{\re^{-\ri\la}}{\Delta(\la)}\int_0^1K(y)\re^{\ri\la y}\D y = 1 - \frac{1}{\Delta(\la)} \sum_{j=1}^2 \alpha^j\hat\kappa(\alpha^j\la),
    \]
    we can rewrite
    \begin{multline*}
        \re^{-\ri\la}n(\la)-\widehat{V}(\la) \\
        = \frac{\re^{-\ri\la}}{\Delta(\la)} \left( \sum_{j=1}^2 \alpha^j\int_0^1K(y) \int_y^1 \re^{-\ri\alpha^j\la (z-y)} V(z) \D z \D y - \int_0^1 K(y) \int_0^y \re^{\ri\la (y-z)} V(z) \D z \D y \right) \\
        + \frac{\widehat{V}(\la)}{\Delta(\la)} \sum_{j=1}^2 \alpha^j \int_0^1\re^{-\ri\alpha^j\la(1-y)}K(y) \D y.
    \end{multline*}
    The derivative of this quantity is
    \begin{multline*}
        \frac{-\re^{-\ri\la}\left( \ri+\frac{\Delta'(\la)}{\Delta(\la)} \right)}{\Delta(\la)} \left( \sum_{j=1}^2 \alpha^j\int_0^1K(y) \int_y^1 \re^{-\ri\alpha^j\la (z-y)} V(z) \D z \D y - \int_0^1 K(y) \int_0^y \re^{\ri\la (y-z)} V(z) \D z \D y \right) \\
        + \frac{\ri\re^{-\ri\la}}{\Delta(\la)} \left( \sum_{j=1}^2 \alpha^{2j}\int_0^1K(y) \int_y^1 \re^{-\ri\alpha^j\la (z-y)} (y-z) V(z) \D z \D y - \int_0^1 K(y) \int_0^y \re^{\ri\la (y-z)} (y-z) V(z) \D z \D y \right) \\
        - \frac{\ri}{\Delta(\la)} \sum_{j=1}^2 \alpha^{2j} \int_0^1\re^{-\ri\alpha^j\la(1-y)}(1-y)K(y) \D y
        + \frac{\sum_{j=1}^2 \alpha^j \hat\kappa(\alpha^j\la)}{[\Delta(\la)]^2} \left( \Delta'(\la)\widehat{V}(\la)-\Delta(\la)\frac{\D}{\D\la}\widehat{V}(\la) \right).
    \end{multline*}
    Similar arguments to those presented in lemma~\ref{lem:FullSoln.OneNonlocal.NondataDecays} justify that this quantity is $\bigoh(1)$ as $\la\to\infty$ along the contours of interest.
    Therefore, the bracketed factor in expression~\eqref{eqn:LongTime.HomogeneousINVP.LinDec.Proof.3} is $\bigoh(1)$.
    Because the exponential in that expression is, uniformly in $x,t$, bounded in modulus, dominated convergence establishes a uniform bound on the remaining integral on the right of equation~\eqref{eqn:LongTime.HomogeneousINVP.LinDec.Proof.2}.
    Hence $I^+(t)=\bigoh(t^{-1})$, uniformly in $x$, as $t\to\infty$.
\end{proof}

\begin{rmk}
    In proposition~\ref{prop:LongTime.HomogeneousINVP.LinDec}, we have not attempted to obtain the leading order behaviour of $v(x,t)$ in the long time limit, only that it decays at least linearly.
    It may even be that further integration by parts reveals uniform quadratic or faster decay of $v$.
    The arguments presented above could be extended to investigate these questions.
\end{rmk}

\begin{rmk}
    One may extend this linear superposition argument to study the case in which the boundary and nonlocal data are not periodic but only asymptotically periodic.
    In this case, the problem for $q$ need not be altered, but the problem for $v$ would have all $h_k$ decaying instead of identically zero.
    How fast must the data decay to their perodic limit to obtain decay of the corresponding solution $v$?
    Questions such as this require a generalisation of proposition~\ref{prop:LongTime.HomogeneousINVP.LinDec} to study the extra terms.
\end{rmk}

\section*{Acknowledgement}

\AckYNCSRP{B. Normatov}
\AckYNCSeed{D. A. Smith}

\bibliographystyle{amsplain}
{\small\bibliography{dbrefs}}

\end{document}